\documentclass[11pt,reqno,oneside]{amsart}
\pdfoutput=1

\usepackage[foot]{amsaddr} 
\usepackage[a4paper]{geometry}
\usepackage{mathtools, amsthm, amsfonts, amssymb}
\usepackage{microtype}
\usepackage[labelfont=bf, font={small}]{caption}[2005/07/16]
\usepackage[colorlinks]{hyperref}
\hypersetup{
  linkcolor=[rgb]{0.3,0.3,0.6},
  citecolor=[rgb]{0.2, 0.6, 0.2},
  urlcolor=[rgb]{0.6, 0.2, 0.2}
}

\usepackage{tikz}
\usetikzlibrary{arrows}
\usetikzlibrary{matrix}
\usetikzlibrary{positioning}
\usetikzlibrary{calc}

\usepackage[nameinlink]{cleveref}
\crefname{theorem}{theorem}{theorems}
\Crefname{theorem}{Theorem}{Theorems}
\crefname{lemma}{lemma}{lemmas}
\Crefname{lemma}{Lemma}{Lemmas}
\crefname{proposition}{proposition}{propositions}
\Crefname{proposition}{Proposition}{Propositions}

\newtheorem{theorem}{Theorem}[section]
\newtheorem{lemma}[theorem]{Lemma}
\newtheorem{proposition}[theorem]{Proposition}

\numberwithin{equation}{section}

\newcommand{\abs}[1]{\lvert #1 \rvert}
\newcommand{\calL}{\mathcal{L}}
\newcommand{\calM}{\mathcal{M}}
\newcommand{\calX}{\mathcal{X}}

\newcommand{\bbP}{\mathbb{P}}
\newcommand{\bfell}{\boldsymbol{\ell}}

\newcommand{\vvirg}{ , \dots , }
\newcommand{\bbk}{\Bbbk}
\newcommand{\dashto}{\dashrightarrow}
\newcommand{\rank}{\mathrm{rank}}
\renewcommand{\bar}[1]{\overline{#1}}

\begin{document}

\title{Identifiability through special linear measurements}

\author[F. Gesmundo]{Fulvio Gesmundo${}^\ast$ }
\thanks{${}^\ast$Institut de Mathématiques de Toulouse; UMR5219 -- Université de Toulouse; CNRS -- UPS, F-31062 Toulouse Cedex 9, France}

\author[A. Grosdos]{Alexandros Grosdos${}^\dagger$}
\thanks{${}^\dagger$Institute of Mathematics, University of Augsburg, 86159 Augsburg, Germany}

\author[A. Uschmajew]{Andr\'e Uschmajew${}^\ddagger$}
\thanks{${}^\ddagger$Institute of Mathematics \& Centre for Advanced Analytics and Predictive Sciences, University of Augsburg, 86159 Augsburg, Germany}

\keywords{linear measurement, algebraic compressed sensing, Noether normalization, low-rank models}
\subjclass[2020]{14Q15, 15A29, 90C30}

\begin{abstract}
    We show that one can always identify a point on an algebraic variety $X$ uniquely with $\dim X +1$ generic linear measurements taken themselves from a variety under minimal assumptions. As illustrated by several examples the result is sharp, that is, $\dim X$ measurements are in general not enough for unique identifiability.
\end{abstract}

\maketitle

\section{Introduction}

Identifying a vector or function from finitely many linear measurements is a central problem in applied mathematics, and plays a particular role in numerical analysis and data science. Often, the object to be identified is assumed to lie in a certain low-parametric model class: this reduces the intrinsic dimensionality of the problem and at the same time increases the identifiability by linear measurements. 

In a finite-dimensional setting, the problem can be formulated as follows. Given a model $\mathcal X \subseteq V$ in a finite-dimensional vector space $V$, one wishes to recover a particular element $x \in \calX$ from $n$ linear measurements
\begin{equation}\label{eq: recovery problem}
y_i = \ell_i(x), \qquad i = 1,\dots,n,
\end{equation}
where $\ell_1,\dots,\ell_n \in V^*$ are linear functionals on $V$. It is then natural to ask whether it is possible to recover $x$ from the given information $y = (y_1 \vvirg y_n)$ and what the minimal number $n$ of required measurements is. Desirably, the number of required measurements should not differ too much from the dimension of $\calX$. Further, one can ask for practical methods for recovering $x$ from $y$ in a reliable and stable way; however, in this work we only focus on the first question.

When $\calX$ is a linear subspace of $V$ one speaks of a linear model. For the sake of simplicity, we may take $\mathcal X = V$. Then the above task is just a generalized interpolation problem and standard linear algebra guarantees that $n = \dim V$ linearly independent functionals are necessary and sufficient to recover $x$ uniquely from $y$. Moreover, given a basis of $V$ the solution $x$ can be computed by solving a system of linear equations. An important example of this kind is the recovery of a polynomial $p \in V = \mathbb R_m[t]$ of degree at most $m$ from the evaluation at $n = m+1$ distinct points $t_0,t_1,\dots,t_m \in \mathbb R$:
\[
y_i = p(t_i), \qquad i=0,1,\dots,m.
\]
We highlight that in this example the set of linear functionals on $V = \mathbb R_m[t]$ that are point evaluations is a nonlinear subset of $V^*$: it is an algebraic variety called the (affine) \emph{rational normal curve} of degree $m$ \cite[Example 1.14]{Harris:AlgGeo}.

Another example of a nonlinear model arising in several application is the set $\calX = \calM_{\le k}$ of matrices of rank at most $k$ in the space $V = \mathbb{R}^{d_1 \times d_2}$ of real matrices of size $d_1 \times d_2$. This is a real algebraic variety of dimension $\dim \calM_{\le k} = (d_1+d_2-k)k$: its defining equations are the minors of size $k+1$, regarded as polynomials in the entries of the matrix. We then wish to recover a matrix $A \in \calM_{\le k}$ from linear measurements
\[
y_i = \ell_i(A) = \langle A, Y_i \rangle, \qquad i=1,\dots,n,
\]
where $\langle \cdot, \cdot \rangle$ is the Frobenius inner product. This problem has been intensively studied in compressed sensing under the name \emph{matrix sensing}; see~\cite{Davenport:2016} for an overview. It has been shown that using 
$n  \geq C \cdot \dim \mathcal X$
random (Gaussian) linear measurements, where $C$ is a constant, a rank-$k$ matrix is determined uniquely from its measurements with very high probability~\cite{Candes:2011}. In addition to this, the linear measurement operator $\bfell : A \mapsto [\ell_i(A)]_{i=1,\dots,n}$ itself satisfies a restricted isometry property, which enables recovery by practical algorithms such as nuclear norm minimization~\cite{Recht:2010,Candes:2011} or iterative hard thresholding~\cite{Jain2010,Goldfarb2011}. If only unique identifiability is required, then $n = \dim \calM_{\le k} + 1$ generic measurements suffice for recovering a generic element $A \in \calM_{\le k}$. This can be explained via classical algebraic geometry, as a consequence of the Noether Normalization Lemma, see e.g.~\cite[Question 7]{BGMV:AlgebraicCompressedSensing}. 

So, the setting of the recovery of matrices of bounded rank is in a way dual to the one of polynomial interpolation. In the polynomial interpolation case, the model is linear and the measurements are taken in a non-linear set $\calL \subseteq V^*$; conversely in the matrix sensing setting, the model $\calX$ is non-linear, an the measurements can be taken in the full space of $V^*$. There are settings where both the model $\calX$ and the set of measurement $\calL$ are non-linear; see \Cref{sec: intro motivating examples}. The goal of this work is to give a version of the Noether Normalization Lemma which applies in such general cases.

\subsection{Problem statement and main result}

We consider the following fundamental problem. Let $\calX \subseteq V$ be an affine algebraic variety of dimension $n$ in a real or complex vector space $V$.  Let $\calL \subseteq V^*$ be an irreducible algebraic variety not contained in a hyperplane of $V^*$. Given $x \in \calX$, we ask: 
\begin{center}
\emph{What is the minimum number~$r$ such that $r$ generic elements of $\calL$ uniquely identify $x$?}
\end{center}
Clearly, for dimension reasons, we should have $r \geq \dim \calX$ in general. The goal of this note is to prove the following result. 
\begin{theorem}\label{thm: main informal}
Let $V$ be a finite-dimensional vector space. Let $\calX \subseteq V$ be an algebraic variety and $x \in \calX$. Let $\calL \subseteq V^*$ be an irreducible algebraic variety not contained in a hyperplane of $V^*$. Then $\dim \calX +1$ generic measurements chosen in $\calL$ uniquely identify~$x$.
\end{theorem}
We point out that \Cref{thm: main informal} holds over every infinite field. The applications are most often over real or complex numbers. 

\subsection{Structured measurement and non-linear models}\label{sec: intro motivating examples}

In this second part of the introduction we discuss in more detail some scenarios from the areas of low-rank methods and polynomial optimization, in which both the model $\calX$ and the set $\calL$ of measurements are non-linear. This discussion is not directly related to the proof of \Cref{thm: main informal}, but serves as a motivation for this work.

In the context of low-rank matrix recovery two notable special cases can be mentioned. The first is \emph{matrix completion}. In this case the measurements are the single entries of the matrix $A \in \calM_{\le k}$, that is, they are of the form
\[
\ell_{\mu,\nu}(A) = a_{\mu \nu} = \langle A, E_{\mu \nu} \rangle
\]
where $E_{\mu \nu}$ is the matrix with one at position $(\mu,\nu)$ and zero elsewhere. In particular, the set of available linear functionals here is finite: in this case \Cref{thm: main informal} does not apply because the variety $\calL$ is not irreducible. In fact, for every fixed finite set of linear functionals, one can always find models $\calX \subseteq \dim V$ for which unique identifiability requires the whole set of $\dim V$ measurements, see e.g.~\cite[Example 12]{BGMV:AlgebraicCompressedSensing}. Nevertheless, it is known that in the matrix completion setting, a logarithmic oversampling with randomly distributed entries in combination with incoherence assumptions allows one (with high probability) to recover $A$ uniquely with convex optimization methods such as nuclear norm minimization~\cite{Candes:2009b,Gross2011,Recht2011}. On the other hand, logarithmic oversampling is not required for the identifiability itself if one instead is allowed to select the entries to be sampled. For instance, the method of cross  approximation, also known as CUR approximation, determines a rank-$k$ matrix from a cross of $\dim \mathcal \calM_{\le k} = (d_1+d_2-k)k$ entries; see \Cref{subsec: cutsInfinity}. 

The second example of structured measurements for recovering a low-rank matrix $A \in \calM_{\le k}$ concerns \emph{rank-one measurements}. In this case, the linear measurements are of the form 
\begin{equation}\label{eq: bilinear measurments}
y_i = \ell_i(A) = \langle A, \xi^{(i)} \otimes \eta^{(i)} \rangle = (\xi^{(i)})^\top A \eta^{(i)}, \qquad \xi^{(i)} \in \mathbb R^{d_1}, \quad \eta^{(i)} \in \mathbb R^{d_2}.
\end{equation}
Such measurements are also called \emph{bilinear measurements} and arise naturally in certain applications based on the \emph{lifting} technique~\cite{Davenport:2016}. Recovery theory for such measurements, especially the design of stable algorithms, is much less developed. Note that in this setting, the measurements come themselves from an algebraic variety, in this case the manifold of (usually normalized) rank-one matrices. Thus, even the basic question of identifiability is not straightforward, since the genericity assumption of the Noether Normalization Lemma is not satisfied. However, \Cref{thm: main informal} guarantees that $n = \dim \calM_{\le k} +1$, generic bilinear measurements (i.e.~for generic choices of $\xi^{(i)},\eta^{(i)}$) are sufficient for unique identifiability of $A \in \calM_{\le k}$. 

The last example shows an interesting connection between matrix recovery and interpolation of polynomials. If in~\eqref{eq: bilinear measurments} $A \in \mathbb R^{d \times d}$ is symmetric and $\xi^{(i)} = \eta^{(i)}$, then the measurements $y_i$ are point evaluations of the quadratic form 
\[
p(\xi) = \langle \xi, A \xi \rangle = \langle A, \xi \otimes \xi \rangle = \sum_{\mu=1}^d \sum_{\nu =1}^d a_{\mu \nu}\xi_\mu \xi_\nu
\]
on $\mathbb R^d$. Therefore recovering a quadratic form from point evaluations $y_i = q(\xi^{(i)})$ is equivalent to recovering the underlying symmetric matrix $A$ from symmetric rank-one measurements. More generally, recovering a (homogeneous) polynomial of degree $m$ in $d$ variables from point evaluations is equivalent to recovering a corresponding symmetric tensor of size $d^{\times m}$ of order $m$ from rank-one measurements. For inhomogeneous polynomials, the setting is similar: one can interpret them as symmetric tensor in $(d+1)^{\times m}$ using a \emph{lifted} vector of the form $(1,\xi)$.

In general, several applications are concerned with recovering polynomials, or equivalently symmetric tensors, from point evaluations. In this setting, one is given an oracle that performs evaluations of a polynomial $p \in V = \mathbb{R}[t_1 \vvirg t_d]_{\leq m}$, and the task is to uniquely determine $p$ from as few evaluations as possible. The evaluation linear functionals $\ell_\xi : p \mapsto p(\xi)$ form an algebraic subvariety of $V^*$, called the (affine cone over the) Veronese variety. With no a priori knowledge, recovery of a polynomial of degree $m$ in $d$ variables requires $\dim V = \binom{m+d}{d}$ evaluations. However, applications are often concerned with polynomials lying on specific subsets of $V$, such as sparse polynomials, polynomials of low Waring rank, or polynomials admitting small (structured) circuits \cite{GKS,KS}. In all these settings \Cref{thm: main informal} applies and guarantees that, if the polynomial $p$ lies in a variety $\calX \subseteq V$ of dimension $n$, then $n+1$ generic evaluations are sufficient to uniquely identify~$p$.

In numerical analysis and approximation theory, multivariate functions are also often discretized using tensor representations. Specifically, for $j=1,\dots,d \ge 3$ let $V^{(j)}$ be an $m_{j}$-dimensional space of univariate real functions on $\Omega^{(j)} \subseteq \mathbb R$ spanned by suitable basis functions $\varphi^{(j)}_1,\dots,\varphi^{(j)}_{m_j}$. Then the tensor product space $V^{(1)} \otimes \dots \otimes V^{(d)}$ contains $d$-variate real functions on $\Omega^{(1)} \times \dots \times \Omega^{(d)} \subseteq \mathbb R^d$ of the form
\begin{equation}\label{eq: tensor product function}
p(\xi) = \sum_{\mu_1 = 1}^{m_1} \cdots \sum_{\mu_d = 1}^{m_d} a_{\mu_1 \mu_2 \cdots \mu_d} (\varphi^{(1)}_{\mu_1} \otimes \cdots \otimes \varphi^{(d)}_{\mu_d})(\xi),
\end{equation}
where $A \in \mathbb R^{m_1 \times \dots \times m_d}$ is the coefficient tensor defining $p$ with respect to the tensor product functions 
\[
(\varphi^{(1)}_{\mu_1} \otimes \cdots \otimes \varphi^{(d)}_{\mu_d})(\xi) =  \varphi^{(1)}_{\mu_1}(\xi_1) \cdots \varphi^{(d)}_{\mu_d}(\xi_d),
\]
which form a basis of $V^{(1)} \otimes \dots \otimes V^{(d)}$. The representation~\eqref{eq: tensor product function} suffers from the curse of dimensionality, since for large $d$ the coefficient tensor $A$ can neither be computed nor  be stored. Numerical tensor methods mitigate this problem by employing low-rank models for the tensor $A$ such as tree tensor network representations;~see the survey articles~\cite{Bachmayr2016,Bachmayr2023} and monographs~\cite{Hackbusch2019,Khoromskij2018} for overview. Then the task of learning $p$ or $A$ in~\eqref{eq: tensor product function} from point evaluations can be of interest. Note that one can write
\begin{equation}\label{eq: tensor product function 2}
p(\xi) = \langle A, \Phi(\xi) \rangle
\end{equation}
where $\Phi \colon \Omega^{(1)} \times \dots \times \Omega^{(d)} \to \mathbb R^{m_1 \times \dots \times m_d}$ maps the point $\xi \in \mathbb R^d$ to the rank-one tensor containing the evaluation of all tensor product basis functions in $\xi$ as entries, that is,
\[
[\Phi(\xi)]_{\mu_1,\dots,\mu_d} = (\varphi^{(1)}_{\mu_1} \otimes \cdots \otimes \varphi^{(d)}_{\mu_d})(\xi).
\]
Hence recovering $p$ in~\eqref{eq: tensor product function} from point evaluations is the same as identifying the coefficient tensor $A$ from the specific rank-one measurements in the image of $\Phi$.

From a perspective of machine learning, $\Phi$ in~\eqref{eq: tensor product function 2} can be interpreted as a particular feature map sending data points $\xi \in \mathbb R^d$ to rank-one tensors $\Phi(\xi)$ in $\mathbb R^{m_1 \times \dots \times m_d}$. The tensor $A$ then becomes the normal vector for a classifying hyperplane in the feature space. Such a construction has been considered in several works, e.g.~\cite{Stoudenmire2016,Novikov2018,Chen2018,Kargas2021,Michel2022}, where again low-rank tensor network models $\calX \subseteq V = \mathbb R^{m_1 \times \dots \times m_d}$ are proposed for learning the high-dimensional tensor~$A$, and ultimately the function $p$ from labelled data $y_i = p(\xi^{(i)}) = \langle A, \Phi(\xi^{(i)}) \rangle$. These low-rank models often form algebraic varieties. Different basis functions $\varphi^{(j)}_{\mu_j}$ for the single variables then lead to different nonlinear families of classifiers. When polynomials are used~\cite{Novikov2018,Chen2018}, e.g., $\varphi^{(j)}_\mu(t) = t^{\mu-1}$ for $\mu=1,\dots,m$ and every $j$, then the available rank-one measurements $\Phi(\xi) =[\xi_1^{\mu_1} \cdots \xi_d^{\mu_d}]$ form a subset dense in an algebraic subvariety~$\mathcal L$ within the variety of rank-one tensors in $\mathbb R^{m \times \dots \times m}$. In an idealized setting of noiseless data one then may ask how many samples are sufficient to exactly learn the coefficient tensor $A \in \calX$ of the classifier from this type of measurements.

\section{Proof of main result}
\subsection{Preliminaries}

The proof of \Cref{thm: main informal} relies on some basic properties of algebraic varieties. The key element that we use is the property that if $Y$ is an \emph{irreducible} algebraic variety of dimension $m$ then any subvariety of $Y$ has dimension at most $m-1$.  This holds both in the affine and projective setting;~see,~e.g.,~\cite[Theorem 1.19]{Shafarevich:1596976}. We assume some basics of algebraic geometry, such as the definition of algebraic varieties, dimension, degree and irreducibility.

We comment here on the notion of genericity. Usually, one says that a property holds \emph{generically} in an irreducible variety $Z$ if it holds for every $z \in U$ where $U$ is a Zariski open subset of $Z$; such open set is often not specified, and in many cases it is not known explicitly. An important fact is that if $Z$ is a subvariety of a Euclidean space, then the closed subset on which a generic property does not hold has measure zero with respect to the Lebesgue measure induced on $Z$; in particular, one can say that if a property holds generically then it holds with probability one, with respect to any probability distribution on $Z$ which is absolutely continuous with respect to the Lebesgue measure.

In our setting, and in particular in the statement of \Cref{thm: main informal}, we say that a collection $\ell_0, \dots, \ell_n$ of measurements in $\calL$ is \emph{generic} with the meaning that $(\ell_0,\dots,\ell_n)$ is generic in $\calL^{\times (n+1)} \subseteq (V^*)^{\times (n+1)}$. In other words, the measurements $(\ell_0 \vvirg \ell_n)$ for which the statement of \Cref{thm: main informal} holds is (more precisely, contains) the complement of a subset of measure zero in $\calL^{\times (n+1)}$. In fact, from the proof of~\Cref{prop:special bertini cuts} below, one can very explicitly construct this set. 

\subsection{Proof for affine setting}

We first prove the main theorem in the affine setting. The proof is largely inspired by projective geometry and we will expand on this point of view in the next subsection. In the following, $\bbk$ is any infinite field and $V$ is a $\bbk$-vector space of finite dimension.

\begin{lemma}\label{lemma: generic is nonconstant}
    Let $\calL \subseteq V^*$ be an irreducible algebraic variety not contained in any hyperplane. Let $v_1,v_2 \in V$. Then the set 
    \[
    C = \{ \ell \in \calL : \ell(v_1) = \ell(v_2) \}
    \]
    is an algebraic variety strictly contained in $\calL$ (possibly empty). In particular, a generic element of $\calL$ takes distinct values on $v_1$ and $v_2$.
\end{lemma}
\begin{proof}
The set $H = \{ \ell \in V^* : \ell(v_1) = \ell(v_2)\}$ is a hyperplane in $V^*$ and $C = \calL \cap H$. Since $\calL$ is an irreducible variety not contained in any hyperplane, we have that $\calL \not \subseteq H$, therefore $C$ is a proper subvariety of $\calL$.
 \end{proof}

Let $\ell \in V^*$ be a linear form and let $y \in \bbk$; then $\ell-y$ is an affine linear form and we write $H(\ell- y)$ for the affine hyperplane that it defines, that is
\[
H(\ell - y) = \{ v \in V : \ell(v) - y = 0\}.
\]

\begin{proposition}\label{prop:special bertini cuts}
    Let $\calX \subseteq V$ be a (possibly reducible) algebraic variety of dimension~$n$. Let $\calL \subseteq V^*$ be an irreducible algebraic variety not contained in any hyperplane. Let $\ell_1 \vvirg \ell_n$ be generic elements of $\calL$ and $y_1, \dots, y_n \in \bbk$ be any set of scalars. Then the set 
    \[
    \calX \cap \{ v \in V : \ell_i(v) = y_i \text{ for all $i = 1 \vvirg n$}\}
    \]
    is finite (possibly empty).
\end{proposition}
\begin{proof}
We proceed by induction on $n = \dim \calX$. In fact, we only require the inequality $\dim \calX \leq n$. If $n =0$, the statement is clear because $\calX$ is itself $0$-dimensional, namely it is a finite set of points. Suppose $n \geq 1$. Let $\calX^{(1)} \vvirg \calX^{(s)}$ be the irreducible components of $\calX$: then $\dim \calX^{(i)} \leq n$ for every $i = 1 \vvirg s$. By~\Cref{lemma: generic is nonconstant}, a generic element $\ell \in \calL$ is non-constant on every $\calX^{(i)}$, unless $\calX^{(i)}$ is a single point. Therefore, for a generic $\ell$ and an arbitrary $y \in \bbk$, if $\dim \calX^{(i)} \geq 1$, then the variety $\calX^{(i)} \cap H( \ell - y )$ is a (possibly empty) proper subvariety of $\calX^{(i)}$, because the restriction of $\ell$ to $\calX^{(i)}$ is not constantly equal to $y$. Since $\calX^{(i)}$ is irreducible, the irreducible components of $\calX^{(i)} \cap H(\ell - y)$ have dimension strictly less than $\dim \calX^{(i)}$. If $\dim \calX^{(i)} = 0$, that is, $\calX^{(i)} = \{p_i\}$ is a single point, then $\calX^{(i)} \cap H( \ell - y)$ is either $\{p_i\}$ or empty, depending on whether $\ell(p_i) = y$.

Now let $\calX_1 = \calX \cap H( \ell_1 - \xi_1)$. We showed that $\calX_1$ is an algebraic variety of dimension at most $n-1$. Therefore the induction hypothesis applies and for generic $\ell_2 \vvirg \ell_n$ and arbitrary $y_2 \vvirg y_n$, we conclude 
\[ 
\calX \cap \{ v \in V : \ell_i(v) = y_i \text{ for all $i = 1 \vvirg n$}\} 
\] 
is a (possibly empty) finite set of points.
\end{proof}

The proof of \Cref{thm: main informal} is now obtained by combining \Cref{lemma: generic is nonconstant} and \Cref{prop:special bertini cuts}.

\begin{proof}[Proof of \Cref{thm: main informal}]
Consider a generic tuple $(\ell_0, \dots, \ell_n)$ of $\calL^{\times(n+1)}$ and let $y_i = \ell_i(x)$. By \Cref{prop:special bertini cuts}, the set 
\[
\calX_0 = \calX \cap \{ v \in V : \ell_i(v) = y_i \text{ for $i= 1 \vvirg n$}\}
\]
is a finite set of points, and by assumption $x$ is one of such points since $y_i = \ell_i(x)$. Further, $\ell_0$ takes distinct values on all points of $\calX_0$ by \Cref{lemma: generic is nonconstant}. Therefore the only element $v$ of $\calX_0$ with the property $y_0 = \ell_0(v)$ is $v = x$. This concludes the proof.
\end{proof}

The proof of \Cref{prop:special bertini cuts} allows one to give a characterization of the genericity condition in \Cref{thm: main informal}. The Zariski open set of $\calL^{\times (n+1)}$ for which the statement holds can be constructed recursively as follows. The linear measurement $\ell_1$ should be chosen so that $H(\ell_1 - \ell_1(x))$ does not contain an irreducible component of $\calX$; since containing a component is a closed condition, this gives an open set $\Omega_1 \subseteq \calL$. The linear form $\ell_2$ should be chosen so that $H(\ell_2 - \ell_2(x))$ does not contain an irreducible component of $\calX \cap H(\ell_1-\ell_1(x))$; similarly to before, the pairs $(\ell_1,\ell_2)$ for which this holds form an open set $\Omega_2 \subseteq \Omega_1 \times \calL \subseteq \calL^{\times 2}$. Continuing this process, one obtains the desired open set.

\subsection{Projective setting}
\label{sec:projective}

We now restate the main result in the projective setting, which provides a variant of the classical Noether Normalization Lemma. 

For a linear subspace $U \subseteq \bbP V^*$, let 
\[
U^\perp = \{ p \in \bbP V : \ell(p) = 0 \text{ for all $\ell \in U$}\}
\]
denote the annihilator of $U$ in $\bbP V$. A straightforward dimension count provides $\dim U = \dim V - \dim U^\perp - 1$.

\begin{theorem}
Let $X \subseteq \bbP V$, $L \subseteq \bbP V^*$ be projective varieties, with $L$ linearly non-degenerate and irreducible and $\dim X = n$. Let $p \in X$. Let $(\ell_0 ,\dots, \ell_{n+1})$ be a generic tuple in $ L^{\times (n+2)}$. Then the rational projection map
\[
\bfell: \bbP V \dashto \bbP (V / \langle \ell_0 , \dots,  \ell_{n+1}\rangle^\perp)
\]
is well-defined on $X$ and $\bfell^{-1}(\bfell(p)) = \{ p \}$.
\end{theorem}

Here, the dashed arrow ``$\dashto$'' indicates that the map $\bfell$ is only defined on a Zariski open subset of $\bbP V$, namely the set $\bbP V \setminus \langle \ell_0 \vvirg \ell_{n+1} \rangle^\perp$; see \cite[Ch. 7]{Harris:AlgGeo}.
\begin{proof}
 In coordinates, we have 
 \[
    \bfell: \bbP V \dashto \bbP ^{n}, \quad
    [v] \mapsto [\ell_0(v) \vvirg \ell_n(v)],
 \] 
 and $\bfell$ is not defined on the linear subspace $\langle \ell_0 , \dots,  \ell_n\rangle^\perp$. To show that $\bfell$ is well-defined on $X$ we need to show that $\langle \ell_0 , \dots,  \ell_n\rangle^\perp \cap X = \emptyset$. This follows from \Cref{thm: main informal} applied to the affine cone $\calX = \hat{X} \subseteq V$ over $X$. Let $v = 0 \in \hat{X}$: we have $\dim \calX = n+1$ and $\ell_0 \vvirg \ell_{n+1}$ are $n+2$ generic linear measurements such that $\ell_i(v) = 0$. By \Cref{thm: main informal}, $v=0$ is the only vector of $\calX$ satisfying $\ell_i(v) = 0$. So for every $v' \in \calX \setminus \{0\}$, we have $\bfell(v') \neq 0$ namely $\langle \ell_0 , \dots,  \ell_n\rangle^\perp \cap X = \emptyset$.

To show that $\bfell^{-1}(\bfell(p)) = \{ p \}$, we apply \Cref{thm: main informal} to a suitable dehomogenization of $X$. Up to reordering, assume $\ell_{n+1}(p) \neq 0$ and let $U \subseteq \bbP V$ be the affine open set $U = \{ [v] \in \bbP V : \ell_{n+1}(v) \neq 0\}$. Moreover, one can choose representatives $w \in V$ for $[w] \in U$ such that $\ell_{n+1}(w) = 1$. Let $\calX = X \cap U$, which is an affine variety of dimension (at most) $n$. Applying \Cref{thm: main informal} to $\calX$ and generic linear forms $\ell_0 \vvirg \ell_n$, we deduce that $p$ is uniquely determined by $\ell_0(p) \vvirg \ell_n(p)$. In turn, this guarantees that $p$ is uniquely determined by $\bfell(p)$ and this concludes the proof.
\end{proof}

We remark here that in the above proof we have twice made the assumption that our choice of linear functionals is generic.
The first one is about the preimage of $0$ via $n+2$ linear functionals being the singleton $\{0\}$.
The second one is about the preimage of $\bfell(p)$.
Therefore our functionals must be chosen in the intersection of two Zariski open sets, 
which is still Zariski open, 
meaning that the result holds generically. 

\section{Sharpness of the result and examples}

The proof of \Cref{thm: main informal} suggests that it is possible that the space $\calX_0$ obtained after $(\dim \calX)$-many affine linear measurements already consists of a single point. This occurs, for instance, in the method of cross approximation for low rank matrices, see \Cref{subsec: cutsInfinity}, and it is always the case if $\calX$ is a linear space. In this section we provide some examples where we can avoid the last cut, but we also show that this is not possible in general. 

First, observe that choosing measurement generically one always obtains, after $\dim \calX$-many linear measurements, a number of points that does not depend on the cuts. This number is the degree of the variety $\calX$, see, e.g., \cite[Ch.~18]{Harris:AlgGeo}.

In principle, one can use additional knowledge on the variety to choose the hyperplanes in a non-generic way so that after $(\dim \calX)$-many linear measurements, rather than finitely many points, one obtains exactly one. However, the theory guarantees that, projectively, one always obtains $\deg \calX$ points counted with multiplicity: algebraically, this is a $0$-dimensional scheme, whose support is the set of points, and the non-reduced structure keeps track of the multiplicities. Therefore, the non-generic choice should be made so that, after $(\dim \calX)$-many measurement, one obtains a $0$-dimensional scheme with most components lying on the hyperplane at infinity and exactly one component lying on the affine chart of interest.

We present here some examples which show that in general we cannot expect to identify elements in $\calX$ with only $\dim(\calX)$ generic linear measurements. 

\subsection{Linear sections with points at infinity}\label{subsec: cutsInfinity}

We begin with a simple example highlighting one of the phenomena that can occur. Let $V = \bbk^2$ and consider the affine variety $\calX = \{ (t,t^2) : t \in \bbk \} \subseteq V$: this is a parabola in the plane and $\dim \calX = 1$. Fix $x = (t_0,t_0^2) \in X$. Let $\ell = \alpha_1 x_1 + \alpha_2 x_2 \in V^*$ and $y = \ell(x)$. For a generic choice of $\ell$, we have $\alpha_2 \neq 0$, and the equation $\ell(t,t^2) - y = 0$ has two solutions: $t_0$ and the other root of the polynomial $\alpha_2 t^2 + \alpha_1 t - y=0$. Therefore, a generic measurement identifies two points on $\calX$. On the other hand, the non-generic choice of $\ell$ with $\alpha_2 =0$ uniquely identifies $x$ because the equation $\ell(t,t^2) - y = 0$ then has a single solution $t_0$. Projectively, the second intersection point of $\calX \cap H( \ell - y)$ is on the line at infinity of $V \subseteq \bbP(V \oplus \bbk)$. We point out that the subvariety $\calL \subseteq V^*$ of linear measurements of the form $\ell = \alpha_1 x_1$ is contained in a hyperplane of $V^*$: in particular, \Cref{thm: main informal} does not even apply, and this phenomenon is unrelated from the statement of the theorem.

A similar phenomenon occurs when considering the cross approximation algorithm, also known as CUR approximation, for low-rank matrices; see, e.g.,~\cite{Goreinov1997,Mahoney2009}. Let $\calX = \{ A \in \bbk^{d_1 \times d_2} \colon \rank(A) \le k\}$ and let $A \in \calX$ be a matrix of with $\rank(A) = k$. For subsets $I \subseteq \{1,\dots,d_1\}$ and $J \subseteq \{1,\dots,d_2\}$ of indices with $\abs{I} = \abs{J} = k$, we denote by $A^I \in \bbk^{k \times d_2}$, $A_J \in \bbk^{d_1 \times k}$ and $A^I_J \in \bbk^{k \times k}$ the submatrices of $A$ obtained by extracting rows and/or columns in $I$ and $J$, respectively. If $A^I_J$ is invertible, one verifies that
\[
A = A_J  (A^I_J)^{-1} A^I.
\]
Therefore, a rank-$k$ matrix $A$ is completely determined by its entries in some \emph{cross} of entries $\Omega = \{(i,j) \colon i \in I \text{ or } j \in J \}$. The number of such entries is $(d_1+d_2-k)k$ which equals the dimension of the variety $\calX$. The entries of $\Omega$ define $(\dim \calX)$-many linear measurements and fixing such entries is equivalent to intersecting $\calX$ with a (highly non-generic) linear space of codimension equal to $\dim \calX$.

We point out that the subvariety $\calL \subseteq V^*$ of linear measurements corresponding to single entries is reducible. In particular, in this setting \Cref{thm: main informal} does not apply, and, similarly to before, this phenomenon is unrelated from the statement of the theorem. In~\cite{Tsakiris2023} algebraic conditions on non-random patterns for identifying low-rank matrices from entries have been studied. Connections with the theory of algebraic matroids and matrix rigidity are outlined in~\cite{KTT,GHIL}.

\subsection{Non-reduced intersections}\label{subsec: nonreduced cuts}

We show with an example a different phenomenon than the one occurring in \Cref{subsec: cutsInfinity}. It is possible, also in the projective setting, that for every $x \in \calX$, there is a particular choice of $\ell_1 \vvirg \ell_n$, with $n = \dim \calX$, such that $\langle \ell_1 \vvirg \ell_n \rangle^\perp \cap \calX$ is a single point. Intuitively, one can design examples for which $\langle \ell_1 \vvirg \ell_n \rangle^\perp$ is a subspace of the tangent space $T_x \calX$ at $x$, and such tangent space only intersects $\calX$ at $x$.

A simple example is the one of conics in the plane. Let $\dim V = 3$ and let $X = \{ x = [x_0,x_1,x_2] \in \bbP V : q(x) = 0\}$ for some homogeneous polynomial $q$ of degree $2$. Then for every $x \in X$ there is exactly one linear form $\ell \in \bbP V^*$ such that $\ell^\perp = T_x X \subseteq \bbP^2$. The intersection $\ell^\perp \cap X$ consists of $2 = \deg X$ points, counted with multiplicity: the condition that $\ell^\perp$ is tangent to $X$ guarantees that the intersection is not reduced at $x$, so the two points coincide. In this case, the intersection $\ell^\perp \cap X$ is a single (non-reduced) point, and the single linear measurement $\ell$ is sufficient to uniquely identify it. We point out that this situation is more artificial than the one of \Cref{subsec: cutsInfinity} because determining $\ell$ requires knowledge not only of the variety $X$ but also of the point $x$ to be recovered. 

\subsection{Minimality of $n+1$ measurement in general}

We show that in general one must require necessarily $\dim \calX + 1$ many linear measurements. This is true for most varieties $\calX \subseteq V$ and $\calL \subseteq V^*$. We show it for some examples: more general examples can be constructed similarly.

For instance, if $X = \{ x = [x_0,x_1,x_2] \in \bbP V : q(x) = 0\}$ is a conic curve in $\bbP V = \bbP^2$, then the discussion of \Cref{subsec: nonreduced cuts} shows that for every $x \in X$ there is a unique $\ell \in \bbP V^*$ such that $\ell^\perp \cap X = \{ x \}$. If $L \subseteq \bbP V$ is any irreducible curve different from the dual variety $X^\vee$ of $X$, then $L \cap X^\vee$ consists of only finitely many points: recovering with a single measurement will only be possible for the finitely many points $x \in X$ such that $T_x X$ is defined by a line $\ell \in L$.

For more general varieties the situation is more complicated. But in general, we expect that for most varieties $\calX \subseteq V$, with $\dim X = n$ only very few points of $\calX$ can be recovered with exactly $n$ linear measurements. 

We illustrate in detail a minimal example in which all but a finite number of $x \in \calX$ require exactly $\dim \calX + 1$ measurements from $V^*$. Fix $\lambda \neq 0,1$ and let $\calX \subseteq \bbk^2$ be the cubic curve 
    \[
    \calX = \{ (x_1,x_2): x_2^2 = x_1(x_1-1)(x_1-\lambda)\}.
    \]
Let $X = \bar{\calX} = \{ (x_0,x_1,x_2) \in \bbP^2 : x_2^2x_0 = x_1(x_1-x_0)(x_1-\lambda x_0) \}$ be its closure in projective space. The intersection of $X$ with the line at infinity $\{x_0 = 0\}$ is the point $p_\infty = (0,0,1)$ with multiplicity $3$. In particular, the line at infinity is tangent to $X$ at $p_\infty$ and $p_\infty$ is an inflection point for $X$. There are eight more inflection points $p_1 \vvirg p_8$ on $X$, given by the eight solutions of the polynomial system
\begin{align*}
0 &= x_2^2 - x_1(x_1-1)(x_1-\lambda); \\ 
0 &= (\lambda (\lambda+1-3x_1)x_1) + (\lambda-(\lambda+1)x_1)^2 + x_2^2 (\lambda+1-3x_1) ;
\end{align*}
the second equation arises as the Hessian determinant of the polynomial defining $X$, after setting $x_0 =1$. Moreover, there are three points $q_1, q_2, q_3 \in \calX$ with the property that the tangent line $T_{q_j} X$ meets $X$ at $p_\infty$. They are $(x_1,x_2) = (0,0),(1,0),(\lambda,0)$. An example where two of the eight inflection points are real is shown in \Cref{fig:cubic}.
\begin{figure}[t]
        \begin{tikzpicture}
            \node[inner sep=0pt] () at (0,0){\includegraphics[width=.3\textwidth]{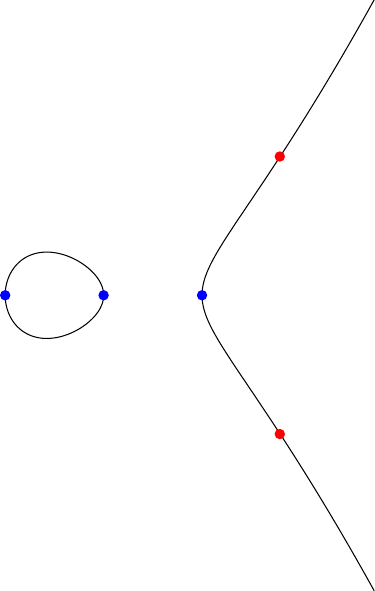}};
            \node[inner sep = 0pt] () at (-.1,0){$q_3$};
            \node[inner sep = 0pt] () at (-.75,0){$q_2$};
            \node[inner sep = 0pt] () at (-2.4,0){$q_1$};
            \node[inner sep = 0pt] () at (1.5,1.7){$p_1$};
            \node[inner sep = 0pt] () at (1.5,-1.7){$p_2$};
        \end{tikzpicture}
\caption{The real part of a cubic curve with $\lambda = 2$. The red points $p_1,p_2$ are two of the nine inflection points. The tangent lines at the three blue points $q_1,q_2,q_3$ are parallel and they meet at the point $p_\infty$. } \label{fig:cubic}
\end{figure}
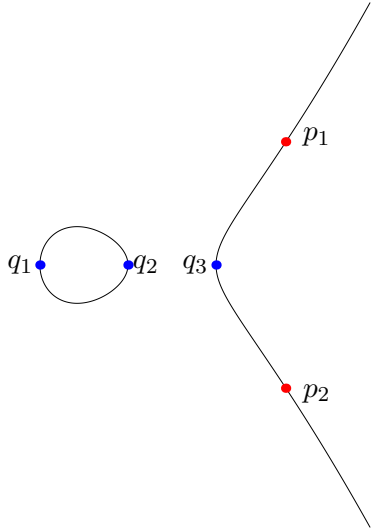

Now, let $p \in \calX$ be any point different from $p_1, \dots, p_8, q_1, q_2,  q_3$. Let $\ell = \alpha_1x_1 + \alpha_2x_2$ be any linear form, let $y = \ell(p)$ and $H (\ell-y) = \{ (x_1,x_2) : \ell(x_1,x_2) - y = 0\}$. We claim that $H (\ell-y) \cap \calX$ always contains at least one point of $\calX$ different from $p$. If $\ell$ is generic, then by B\'ezout's Theorem \cite[Ch.~18]{Harris:AlgGeo}, $H (\ell-y) \cap \calX$ consists of three distinct points, one of them being $p$. 

There are four special (non-generic) choices of $\ell$ such that $H (\ell-y) \cap \calX$ is tangent at some other point of $X = \bar{\calX}$. These lines can be constructed as follows. The dual curve $X^\vee \subseteq  \bbP V^*$ of $X$, parametrizing all tangent lines to $X$, has degree $6$. The point $p$ corresponds to a line $p^\perp \subseteq \bbP V^*$, which is tangent to $X^\vee$ at the point corresponding to $T_pX$ and intersects $X^\vee$ at four more points. These four points correspond to four lines in $\bbP V$ which are tangent to $X$ at four points. Note that none of these points is the point at infinity $p_\infty$ because the tangent line to $X$ at $p_\infty$ is the line at infinity itself. Hence for these special lines $H(\ell - y)$ the intersection $H (\ell_j-y) \cap \calX$ consists of $p$ and the point to which $H(\ell - y)$ is tangent to $X$, which are two points of $\calX$.

Finally, if $H (\ell-y)$ is the tangent line to $\calX$ at $p$, then $H (\ell-y) \cap \calX$ has to contain one more point, because $p$ is not an inflection point nor a point whose tangent line intersects $\bar{\calX}$ at infinity.

This shows that no linear form $\ell$ uniquely determines $p$ in $X$. On the other hand, since $\dim \calX = 1$, \Cref{thm: main informal} guarantees that two generic enough linear forms $\ell_1, \ell_2$ will uniquely determine $p$.

For curves of higher degree, one can prove that under mild genericity assumptions on the curve $\calX \subseteq \bbk^2$, no point on $\calX$ can be recovered with a unique measurement. This relies on some classical facts on the geometry of the dual curve. We omit the full proof and we only sketch the argument. Let $X = \{ f = 0\} \subseteq \bbP^2$ be an irreducible curve, where $f \in \bbk[x_0,x_1,x_2]_d$ is a generic homogeneous polynomial of degree $d \geq 4$. Then the singularities of the dual curve of $X$ are simple nodes and simple cusps \cite[Prop.~1.2.4]{GKZ}: in particular, for every $p \in X$, and every line $L \subseteq \bbP^2$ passing through $p$, the intersection $ L \cap X$ contains at least two points other than $p$.

Let $L_\infty = \{ x_0 = 0\} \subseteq \bbP^2$ and $\calX = X \setminus L_\infty \subseteq \bbk^2 = \bbP^2 \setminus L_\infty$ be the affine curve defined by the dehomogenization $f|_{x_0 = 1}$ of the polynomial $f$. Fix $p \in \calX$, a linear form $\ell$, and set $y = \ell(p)$. The fact stated above then guarantees that the intersection $H(\ell - y) \cap \calX$ contains at least one point different from $p$.  In other words, $p$ cannot be uniquely determined with a single linear measurement.

\subsection*{Acknowledgments} We thank Paul Breiding, Giorgio Ottaviani and Nick Vannieuwenhoven for helpful discussions and for pointing out several relevant examples. 

This work was partially supported by the Thematic Research Programme ``Tensors: geometry, complexity and quantum entanglement'', University of Warsaw, Excellence Initiative -- Research University and the Simons Foundation Award No. 663281 granted to the Institute of Mathematics of the Polish Academy of Sciences for the years 2021--2023. The work of A.U.~was supported by the Deutsche Forschungs\-gemeinschaft (DFG, German Research Foundation) – Projektnummer 506561557.

{\small
\bibliographystyle{alphaurl}
\bibliography{linearcuts.bib}
}

\end{document}